\documentclass[twoside,centertags]{amsart}

\newtheorem{innercustomthm}{Theorem}
\newenvironment{customthm}[1]
  {\renewcommand\theinnercustomthm{#1}\innercustomthm}
  {\endinnercustomthm}
  
\usepackage{fancyhdr}           
\fancypagestyle{plain}{
    \lhead{}
    \fancyhead[R]{\thepage}
    \fancyhead[L]{}
    
    \fancyfoot{}
}
\usepackage{comment}

\usepackage[cmtip,color,all]{xy}
\usepackage{amsmath}
\usepackage{amsfonts}
\usepackage{amssymb}
\usepackage{amsthm}
\usepackage{graphicx}
\usepackage{mathrsfs}
\usepackage{hyperref}
\usepackage{tikz-cd}
\usepackage{xcolor}
\usepackage{bbold}
\usepackage{pifont}
\usepackage{mathtools}

\newtheorem{theorem}{Theorem}[subsection]
\newtheorem{corollary}[theorem]{Corollary}

\newtheorem{proposition}[theorem]{Proposition}
\newtheorem*{theorem*}{Theorem}

\theoremstyle{definition}
\newtheorem{definition}[theorem]{Definition}

\theoremstyle{definition}
\newtheorem{remark}[theorem]{Remark}

\usepackage{geometry}
\usepackage{setspace}
\usepackage{stmaryrd,Commons,xcolor,soul, tikz}
\usetikzlibrary{cd}
\usepackage{circuitikz}

\usepackage[style = alphabetic, doi = false, url=false, isbn = false]{biblatex}
\usepackage{hyperref}
\usepackage{xurl}
\hypersetup{breaklinks=true}
\renewbibmacro{in:}{}
\DeclareFieldFormat[article]{title}{\mkbibemph{#1}}
\DeclareFieldFormat[misc]{title}{\mkbibemph{#1}}
\DeclareFieldFormat{journaltitle}{#1}
\definecolor{winered}{rgb}{0.5,0,0}
 \usepackage{hyperref}
 \hypersetup{citecolor = winered, linkcolor = winered, menucolor = winered, colorlinks = true}

\allowdisplaybreaks
\graphicspath{{./images/}}

\newcommand{\F}{\mathbb{F}}
\newcommand{\G}{\mathbb{G}}
\newcommand{\sfC}{\mathsf{C}}

\newcommand{\s}{\mathsf{Set}}
\newcommand{\sS}{\mathsf{sSet}}
\newcommand{\Coalg}{\mathsf{Coalg}}
\newcommand{\CoCo}{\mathsf{CoCoalg}}
\newcommand{\sC}{\mathsf{sCoalg}}
\newcommand{\sCC}{\mathsf{sCoCoalg}}
\newcommand{\dgC}{\mathsf{dgCoalg}}
\newcommand{\dgA}{\mathsf{dgAlg}}
\newcommand{\Og}{\mathcal{O}_G}

\makeatletter
\newcommand{\colim@}[2]{%
  \vtop{\m@th\ialign{##\cr
    \hfil$#1\operator@font colim$\hfil\cr
    \noalign{\nointerlineskip\kern1.5\ex@}#2\cr
    \noalign{\nointerlineskip\kern-\ex@}\cr}}%
}
\newcommand{\colim}{%
  \mathop{\mathpalette\colim@{\rightarrowfill@\scriptscriptstyle}}\nmlimits@
}
\renewcommand{\varprojlim}{%
  \mathop{\mathpalette\varlim@{\leftarrowfill@\scriptscriptstyle}}\nmlimits@
}
\renewcommand{\varinjlim}{%
  \mathop{\mathpalette\varlim@{\rightarrowfill@\scriptscriptstyle}}\nmlimits@
}
\makeatother

\theoremstyle{definition}

\title{Equivariant homotopy theory via simplicial coalgebras}

\author{Sofía Martínez Alberga}
\address{
Sofía Martínez Alberga\\
Purdue University\\Department of Mathematics\\ 150 N. University St. West Lafayette, IN 47907, USA
}
\email{mart1789@purdue.edu}

\author{Manuel Rivera}
\address{
Manuel Rivera\\
Purdue University\\Department of Mathematics\\ 150 N. University St. West Lafayette, IN 47907, USA
}
\email{manuelr@purdue.edu}

\begin{document}
\singlespacing
\maketitle
\begin{abstract}{
 Given a commutative ring, $R$, a $\pi_1$-$R$-equivalence is a continuous map of spaces inducing an isomorphism on fundamental groups and an $R$-homology equivalence between universal covers.
 When $R$ is an algebraically closed field, Raptis and Rivera described a full and faithful model for the homotopy theory of spaces up to $\pi_1$-$R$-equivalence.
 They use simplicial coalgebras considered up to a notion of weak equivalence created by a localized version of the Cobar functor.
 In this article, we prove a $G$-equivariant analog of this statement using a generalization of a celebrated theorem of Elmendorf.
 We also prove a more general result about modeling $G$-simplicial sets considered under a linearized version of quasi-categorical equivalence in terms of simplicial coalgebras. 
}
\end{abstract}

\section{Introduction}

Equivariant algebraic topology is the study of algebraic invariants of topological spaces equipped with an action of a fixed group, $G$, also known as $G$-\textit{spaces}.
In equivariant homotopy theory, one considers $G$-equivariant continuous maps between $G$-spaces up to the appropriate notion of homotopy.  The present article is motivated by answering the general question:

\begin{quote}
\textit{What is a $G$-equivariant homotopy type in terms of chain level (co)algebraic structure?}
\end{quote}

We give an answer to a version of this question when the coefficients lie in an algebraically closed field of arbitrary characteristic. To do this, we use the framework of simplicial cocommutative coalgebras considered under an appropriate notion of weak equivalence. We also analyze the case when coefficients are taken over an arbitrary perfect field. 

By an equivariant version of the Whitehead's theorem, a map $f \colon X \to Y$ between $G$-CW complexes is a $G$-homotopy equivalence if and only if $f$ is a $G$-weak homotopy equivalence, i.e. for all subgroups $H$ of $G$ the induced maps
\[\pi_n(f^H) \colon \pi_n(X^H,x) \to \pi_n(Y^H,f(x)),\]
on all homotopy groups between $H$-fixed point spaces are isomorphisms.
We shall consider a generalization of this notion for any commutative ring $R$.
For technical reasons, we use the framework of reduced simplicial sets (simplicial sets with a single vertex). 
First, we say a map $f \colon X \to Y$ of reduced simplicial sets is a $\pi_1$-$R$-\textit{equivalence} if
\[\pi_1(f) \colon \pi_1(|X|)\to \pi_1(|Y|),\] and 
\[H_n(\widetilde{|f|};R) \colon H_n(\widetilde{|X|};R) \to H_n(\widetilde{|Y|}; R), \text{ for all $n \geq 0$},\] 
are isomorphisms, where $|-|$ denotes the geometric realization functor, $\pi_1$ the fundamental group, $\widetilde{|f|} \colon \widetilde{|X|} \to \widetilde{|Y|}$ the map induced by $f$ at the level of universal covers, and $H_n( -;R)$ the $n$-th homology functor with $R$-coefficients. 
We promote this notion to the $G$-equivariant setting by defining a $G$-equivariant map $f \colon X \to Y$ of $G$-simplicial sets to be a $G$-$\pi_1$-$R$-\textit{equivalence} if the induced map $f^H \colon X^H \to Y^H$ on all simplicial sets of $H$-fixed points is a $\pi_1$-$R$-equivalence for all subgroups $H$ of $G$. 
Notice that a $\pi_1$-$\mathbb{Z}$-equivalence is a weak homotopy equivalence, so a $G$-$\pi_1$-$\mathbb{Z}$-equivalence is a $G$-weak homotopy equivalence.

Our main result is that, when $R=\mathbb{F}$ is an algebraically closed field, a $G$-equivariant version of the $\mathbb{F}$-chains functor induces a full and faithful embedding from the homotopy theory of $G$-simplicial sets up to $G$-$\pi_1$-$\F$-equivalence into a suitable homotopy theory of $G$-connected simplicial cocomutative $\F$-coalgebras. In particular, any $G$-reduced simplicial set may be recovered functorially, up to $G$-$\pi_1$-$\F$-equivalence, from its simplicial coalgebra of chains. To state this precisely, we recall some (co)algebraic notions. 

A \textit{connected  simplicial cocommutative $R$-coalgebra} is a simplicial object in the category of cocommutative coassociative $R$-coalgebras which has a single generator in degree $0$. An example of such an object is the simplicial coalgebra of chains on a simplicial set defined as follows. 
\\

\noindent \textbf{Definition.}
 For any commutative ring with unit $R$ and any simplicial set $X$, define the \textit{simplicial coalgebra of $R$-chains} $R[X]$ to be the simplicial coalgebra given in degree $n$ by the free $R$-module generated by the set $X_n$ of $n$-simplices with coproduct induced by linearly extending the diagonal map $\sigma \mapsto \sigma \otimes \sigma$ on any $\sigma \in X_n$. The face and degeneracy maps of $R[X]$ are induced by those in $X$. If $G$ is a discrete group, a $G$-action on a simplicial set $X$ induces a natural $G$-action on the simplicial coalgebra of chains $R[X]$ giving rise a functor
\[ R_G[ - ] \colon G\text{-}\mathsf{sSet}_0 \to G\text{-}\mathsf{sCoCoalg}_R^0,\]
from the category of $G$-reduced simplicial sets and $G$-connected simplicial cocommutative coalgebras.
\\

 This functor has a right adjoint denoted by
$\mathcal{P}_G \colon G\text{-}\mathsf{sCoCoalg}_R^0 \to G\text{-}\mathsf{sSet}_0.$
The normalized chains construction takes a connected simplicial cocommutative coalgebra $C$ and produces a differential graded (dg) coaugmented coassociative $R$-coalgebra $\mathcal{N}_*(C)$.
A morphism $f \colon C \to D$ of connected simplicial cocomutative $R$-coalgebras is said to be a $\mathbb{\Omega}$-\textit{quasi-isomorphism} if the induced map of dg coaugmented coalgebras
$\mathcal{N}_*(f) \colon \mathcal{N}_*(C) \to \mathcal{N}_*(D)$ becomes a quasi-isomorphism of dg algebras after applying the \textit{Cobar construction}, a classical functor that produces a dg (free) algebra given a coaugmented dg coalgebra. 
 
We consider a localized version of this notion from \cite{RR22}, called $\widehat{\mathbb{\Omega}}$-\textit{quasi-isomorphism}, which involves localizing a particular set of degree $0$ elements after applying the Cobar construction. Under suitable hypotheses, any $\mathbb{\Omega}$-quasi-isomorphism is an $\widehat{\mathbb{\Omega}}$-quasi-isomorphism and any $\widehat{\mathbb{\Omega}}$-quasi-isomorphism is a quasi-isomorphism and these inclusions are strict. These notions can be understood in the $G$-equivariant setting by considering fixed points with respect to all subgroups: a $G$-equivariant morphism $f \colon C \to D$ of $G$-connected simplicial cocomutative $R$-coalgebras is said to be a $G\text{-}\mathbb{\Omega}$-\textit{quasi-isomorphism} if the map on $H$-fixed points $f \colon C^H \to D^{H}$ is a  $\mathbb{\Omega}$-\textit{quasi-isomorphism} for all subgroups $H$ of $G$. 
The notion of $G$-$\widehat{\mathbb{\Omega}}$-\textit{quasi-isomorphism} is defined similarly. Our main theorem, in the language of model category theory, is the following. 
\newpage
\begin{customthm}{1}\label{one} Let $G$ be a group and $R$ a commutative ring. 
\begin{enumerate}
\item There exists a model category structure on $G$-$\mathsf{sSet}_0$ with weak equivalences being $G$-${\pi_1}$-$R$-equivalences and cofibrations generated by the collection of all morphisms 
\[\underset{G/H}{\sqcup} i \colon \bigsqcup_{G/H} X \hookrightarrow \bigsqcup_{G/H} Y,\] where $i \colon X \hookrightarrow Y$ a monomorphism of reduced simplicial sets, $H$ a subgroup of $G$, and the $G$-action on $\bigsqcup_{G/H} X$ and $\bigsqcup_{G/H} Y$ is induced by the action of $G$ on the set of left cosets $G/H$. 
\item When $R=\mathbb{F}$ is a field, there exists a model category structure on $G$-$\mathsf{sCoCoalg}_{\mathbb{F}}^0$ with weak equivalences being $G$-$\widehat{\mathbb{\Omega}}$-quasi-isomorphisms and cofibrations generated by the collection of all morphisms 
\[\underset{G/H}{\oplus} j \colon \bigoplus_{G/H} C \hookrightarrow \bigoplus_{G/H} D,\] where $j \colon C \hookrightarrow D$ a injection of connected simplicial cocommutative $\mathbb{F}$-coalgebras, $H$ a subgroup of $G$, and the $G$-action on $\bigoplus_{G/H} C$ and $\bigoplus_{G/H} D$ is induced by the action of $G$ on the set of left cosets $G/H$. 
\item The adjunction
\[\F_G[-]: G\text{-}\sS_0 \rightleftarrows G\text{-}\sCC_\F^0: \mathcal{P}_G,\]
becomes a Quillen adjunction between the model category structures of (1) and (2).
\item If $\F$ is an algebraically closed field, $\F_G[-]$ is homotopically full and faithful. More generally, if $\F$ is a perfect field, the derived unit of the Quillen adjunction in (3) can be identified with the natural transformation $X \mapsto \delta_G(X)^{h\mathbb{G}}$, where $\mathbb{G}$ denotes the absolute Galois group of $\F$, $\delta_G(X)$ the $G$-simplicial set $X$ equipped with the trivial $\mathbb{G}$-action, and $(-)^{h \mathbb{G}}$ the homotopy $\mathbb{G}$-fixed points interpreted appropriately.
\end{enumerate}
\end{customthm}
There are three main components to the proof of the above theorem.
\begin{enumerate}
    \item Theorem \ref{one} is the $G$-equivariant analogue of a result of the second author and G. Raptis. It was shown in  \cite{RR22} that the simplicial coalgebra of chains $\F[-] \colon \mathsf{sSet}_0 \to \mathsf{sCoCoalg}^0_{\F}$ defines a left Quillen functor between model structures for reduced simplicial sets up to $\pi_1$-$\F$-equivalence and connected simplicial cocommutative coalgebras up to $\widehat{\mathbb{\Omega}}$-quasi-isomorphism and is homotopically full and faithful when $\F$ is algebraically closed. A similar result for two other pairs of model categories on $\mathsf{sSet}_0$ and $\mathsf{sCoCoalg}^0_{\F}$ is also shown in \cite{RR22}: the first treats $\mathsf{sSet}_0$ under an $\F$-lineraized version of quasi-categorical equivalence and $\mathsf{sCoCoalg}^0_{\F}$ under $\mathbb{\Omega}$-quasi-isomorphisms; the second treats $\mathsf{sSet}_0$ under $\F$-homology equivalence and $\mathsf{sCoCoalg}^0_{\F}$. We also prove versions of Theorem \ref{one} in these two other settings.
    \item To prove $(1)$ and $(2)$ of Theorem \ref{one}, we apply a criteria for the existence of the \textit{fixed point model structure} on the $G$-equivariant category $G$-$\mathsf{C}$ given a cofibrantly generated model category $\mathsf{C}$, see \cite{Ste16}.  
    \item To prove $(3)$ and $(4)$ of Theorem \ref{one}, we use a generalization of a theorem of Elmendorf that allows us to interpret the fixed point model structure on $G$-$\mathsf{C}$, for suitable model categories $\mathsf{C}$, as the projective model structure on the category of functors $\mathcal{O}_G^{op} \to \mathsf{C}$, where $\mathcal{O}_G$ denotes the orbit category of the group $G$, see \cite{Elm83, pia91}.  
\end{enumerate}

The present text continues a line of research that started with promoting algebraic models for the rational homotopy theory of simply connected spaces to the $G$-equivariant setting, see  \cite{Tri82, Scu08}. In \cite{MaSc02}, the authors build upon \cite{Ma01} to describe a full and
faithful model for the homotopy theory of $G$-simply connected spaces of finite type under
$\mathbb{F}_p$-equivalence in terms of diagrams of dg $E_{\infty}$-$\mathbb{F}_p$-algebras over the orbit category of $G$. In the present work, we remove any restrictions on the fundamental group as well as the finite type hypotheses by building upon \cite{RR22}. We also consider $G$-simplicial sets under a linearized notion of quasi-categorical equivalence and describe a model for these in terms of simplicial coalgebras.

The intended future applications for the present work are similar to those found in rational homotopy theory. For example, once one has a chain level structure that naturally characterizes $G$-spaces, the next step is to develop an algebraic theory that allows for extracting the ``smallest possible" model in the weak equivalence class of such structure from which one can efficiently compute topological and geometric invariants. 

\subsection*{Acknowledgements}
The first author would like to acknowledge support of NSF grant DGE-1842166. The second author acknowledges support by NSF Grant DMS 2405405 and by Shota Rustaveli National Science Foundation of Georgia (SRNSFG) [grant number FR-23-5538].

\section{Preliminaries}
In this section we recall the relevant non-equivariant notions and results from \cite{RR22} that will be promoted to the equivariant context in subsequent sections.

\subsection{Simplicial Sets}
Let $\mathbb{\Delta}$ be the category whose objects are the non-empty finite ordinals $\{ [n]  =\{0<\ldots < n\} \mid n \in \mathbb{N}\}$ and morphisms are given by order preserving maps. 
Recall that any morphism in the category can be written as a composition of coface maps $d^i : [n-1] \rightarrow [n]$ and codegeneracy maps $s^i: [n+1] \rightarrow [n]$ for $i = 0,\ldots,n. $
A \emph{simplicial object in a category, $\sfC$}, is a functor $F: \mathbb{\Delta}^{op} \rightarrow \sfC$. 
Simplicial objects in $\sfC$ form a category, denoted by $\mathsf{sC}$, with natural transformations as morphisms. 
For any $F \in \mathsf{sC}$, we write $F_n := F([n])$.

For the purposes of this paper, we will focus our attention to simplicial objects in the category sets, denoted $\sS$, and in the category of coalgebras over a ring $R$, denoted $\sC_R$.
We will discuss the latter the next subsection.
We denote by $\sS_0$ the full subcategory of $\sS$ whose objects are all simplicial sets $S$ such that $S_0$ is a singleton.The objects of $\sS_0$ are called \emph{reduced simplicial sets}.

\begin{definition}\label{def2.1} Let $R$ be a commutative ring.
We say a map of reduced simplicial sets $f : X \rightarrow Y$ is a: 
\begin{enumerate}
    \item \textit{categorical $R$-equivalence} if it becomes a quasi-isomorphism of dg $R$-algebras after applying the functor $\Lambda : \sS_0 \rightarrow \dgA_R$, where  $\Lambda$ is the restriction to $\sS_0$ (the ``one-object" setting) of the functor $\sS \to \mathsf{dgCat}_R$ that is the left adjoint of the dg nerve functor $\mathsf{N}_{\mathsf{dg}} : \mathsf{dgCat}_R \rightarrow \sS$ from dg categories to simplicial sets (see \cite[Construction 1.3.1.6]{ha})
    \item \textit{$\pi_1$-$R$-equivalence} if it induces an isomorphism between the fundamental groups of the geometric realization (i.e. $\pi_1(|f|)$ is an isomorphism) and an $R$-homology isomorphism 
    \[H_*(\widetilde{|f|};R) \colon H_*(\widetilde{|X|};R) \xrightarrow{\cong} H_*(\widetilde{|Y|};R),\]
    between universal covers of the geometric realization, and
    \item \textit{$R$-equivalence} if it induces an isomorphism
   \[ H_*(f;R) \colon H_*(X;R) \xrightarrow{\cong} H_*(Y;R),\]
    in $R$-homology.
\end{enumerate}
\end{definition} 
The first notion above can be thought of as a ``linearized'' version of categorical equivalence ( Joyal equivalence) between reduced simplicial sets. When $R=\mathbb{Z}$, the second notion coincides with that of weak homotopy equivalence of the simplicial sets. 
It is worth noting that there is a result that relates these three definitions.
\begin{proposition} (Proposition 5.3.6 in \cite{RR22}
    The classes of morphisms in $\sS_0$ described in (1),(2), and (3), in Definition \ref{def2.1}, satisfy the following strict inclusions, 
    $$(1) \subset (2) \subset (3).$$
\end{proposition}
 
Furthermore, these notions give rise to three model category structures.

\begin{theorem}\label{t2.1.3} (Theorem 1 in \cite{RR22})
    Let $R$ be a commutative ring. 
    The category $\sS_0$ admits three left proper combinatorial model category structures, denoted by ($\sS_0$, $R$-cat-eq.), ($\sS_0$, $\pi_1$-$R$-eq.), and ($\sS_0$, $R$-eq.) with the weak equivalences listed above, respectively, and monomorphisms as cofibrations.
\end{theorem}
For most of this text, we consider the case when $R$ is a field $\F$ and to simplify notation we denote $\mathbb{F}$-cat-eq. as ``$\simeq_1$'', $ \pi_1\text{-}\mathbb{F}$-eq. as ``$\simeq_2$'', and $\mathbb{F}$-eq. as ``$\simeq_3$''.

\subsection{Simplicial Coalgebras}

Fix a commutative ring $R$ and we write $\otimes$ to mean $\otimes_R$. 
Denote the category of countial coassociative $R$-coalgebras as $\Coalg_R$. 
The objects in $\Coalg_R$ are triples $(C, \Delta, \epsilon)$ with $C$ being an $R$-module, $\Delta:C \rightarrow C \times C$  being the coassociative coproduct of $C$, and  $\epsilon: C \rightarrow R$ being the counit of $C$.
The morphisms in this category are $R$-linear maps preserving the coproduct and counit. 
We denote by $\CoCo_R$ the category of cocommutative $R$-coalgebras. 
This is the full subcategory of $\Coalg_R$ consisting of those coalgebras $(C,\Delta,\epsilon)$ for which the coproduct $\Delta$ is cocommutative, i.e., it satisfies the equation $\Delta = \tau \circ \Delta$, where the $R$-linear map $\tau : C \otimes C \rightarrow C \otimes C$ is the switch map.

A simplicial (cocommutative) coalgebra is a simplicial object in $\Coalg_R$ ($\CoCo_R$). Explicitly, a simplicial coalgebra is a functor $C : \Delta^{op} \rightarrow \Coalg_R$, but can also be described as a simplicial $R$-module, $C$, equipped with maps of simplicial $R$-modules
$$\Delta : C \rightarrow C \otimes C \quad \text{and} \quad  \epsilon: C \rightarrow R,$$
where $(C \otimes C)_n$ = $C_n \otimes C_n$. 
We say a simplicial coalgebra is \emph{connected} if $C_0 = (R, \Delta_0, \text{id}_R)$ where \(\Delta_0: R \xrightarrow{\cong} R \otimes R \) is given by \(\Delta_0(1) = 1 \otimes 1\).
For the remainder of this text, denote by $\sC^0_R$ the full subcategory of $\sC_R$ consisting of connected simplicial coalgebras and define the category $\sCC^0_R$ similarly.
Analogously, we make the following definitions.

\begin{definition}\label{def2.2.1} 
We say a map of simplicial connected coalgebras $f : C \rightarrow C'$ is a:
\begin{enumerate}
    \item \textit{$\mathbb{\Omega}$-quasi-isomorphism} if the induced map dg algebras, $\mathbb{\Omega}(f)$, is a quasi-isomorphism of dg algebras, where $\mathbb{\Omega} := \mathsf{Cobar} \circ \mathcal{N}_*$ (described in Figure \ref{functors}).
    \item \textit{$\widehat{\mathbb{\Omega}}$-quasi-isomorphism} if the induced map of dg algebras, $\widehat{\mathbb{\Omega}}(f)$, is quasi-isomorphism of augmented dg algebras, where $\widehat{\mathbb{\Omega}} : = \mathcal{L}_{\iota} \circ \mathfrak{X}$ (described in Figure \ref{functors}). 
    \item \textit{quasi-isomorphism} if the induced map on normalized chains, $\mathcal{N}_*(f)$, induces an isomorphism on homology.
\end{enumerate}
\end{definition} 
In Figure \ref{functors}, we have summarized some of the functors in Definition \ref{def2.2.1}. For a precise definition of these functors we refer to \cite{RR22}; for the sake of completeness, we provide a short description of these with emphasis on intuition.
\begin{figure}[h]
\centering
\begin{tikzcd}[ampersand replacement=\&]
  \& \sCC_{\F}^{0} \arrow[ldd, "\mathbb{\Omega}", shift left=2] \arrow[ld, "\mathcal{N}_*"] \arrow[dd, "\mathfrak{X}", shift left=2] \arrow[rr, "\widehat{\mathbb{\Omega}}"] \&  
  \& \dgA_{\F} \\
\dgC_{\F} \arrow[d, "\mathsf{Cobar}"', shift right=2]   \&   \&  \&              \\
\dgA_{\F}  \& \dgA_{\F}^{+} \arrow[rruu, "\mathcal{L}_{\iota} "]   \&  \&         
\end{tikzcd}
\caption{Functors in Definition \ref{def2.2.1}}\label{functors}
\end{figure}
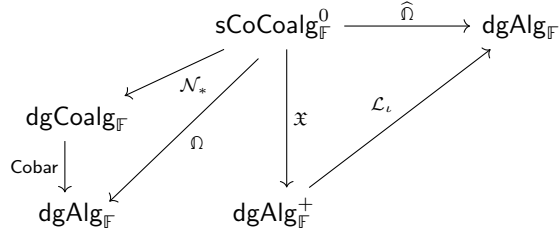
The arrow labeled $\mathcal{N}_*$ is the usual normalized chains functor which takes a simplicial 
coalgebra, $(C, \Delta, \epsilon)$, and produces a conilpotent dg coassociative coalgebra
$\mathcal{N}_*(C)$ with differential given by the alternating sum of face maps of $C$ and coproduct induced by the Alexander-Whitney diagonal approximation. The $\mathsf{Cobar}$ functor builds a dg algebra from a conilpotent dg coalgebra by first modding out by the cokernel of the coaugmentation, shifting the grading of the underlying dg module down by one, and then taking the tensor algebra. The differential is then defined by extending the sum of differential and coproduct of $\mathcal{N}_*(C)$ as a derivation.
It is worth noting that $\mathsf{Cobar}$ has a right adjoint known as the $\mathsf{Bar}$ functor given by the Bar construction. 

For any connected simplicial \textit{cocommutative} coalgebra $C$, the dg algebra $\mathbb{\Omega}(C)$ carries a natural dg \textit{bialgebra} structure. 
We denote the coproduct of $\mathbb{\Omega}(C)$ by $\nabla$. The arrow labeled by $\mathfrak{X}$ is a functor that produces a \textit{marked} dg algebra (a dg algebra equipped with a set of elements) given a connected simplicial cocommutative coalgebra. 
This functor is given by $\mathfrak{X}(C) = (\mathbb{\Omega}(C),P_C)$, where $P_C$, the set of representatives of the ``monoid-like" elements of the bialgebra $H_0(\mathbb{\Omega}(C))$.
Explicitly, \(P_C = \{a \in \mathbb{\Omega(C)}_0 : H_0(\nabla)([a]) = [a] \otimes [a], \epsilon([a]) = 1\}.\)
The functor $\mathcal{L}_{\iota}$ can be thought of as the derived localization of the dg algebra $\mathbb{\Omega(C)}$ at a marked set of elements. (The $\iota$ in the notation refers to a particular replacement to present this derived localization as a functorial model). 

Similarly as before, we have a result relating these three notions of  maps. 

\begin{proposition} (Proposition 5.3.5 in \cite{RR22})
For $R = \F$ be a field. 
    The classes of morphisms in $\sS_0$ described in (1),(2), and (3), in Definition \ref{def2.2.1}, satisfy the following strict inclusions:
    $$(1) \subset (2) \subset (3).$$
\end{proposition}.

We recall a similar result to that of Theorem \ref{t2.1.3}.
\begin{theorem}\label{t2.2.2} (Theorem 2 in \cite{RR22})
     Let $\F$ be a field. 
     The category $\sCC^0_\F$ admits
three left proper combinatorial model category structures, denoted by ($\sCC^0_\F$, ${\mathbb{\Omega}}$-q.i.),
($\sCC^0_\F$, $\widehat{\mathbb{\Omega}}$-q.i.), and ($\sCC^0_\F$, q.i.), which have the injective maps as cofibrations.
\end{theorem}
We further simplify notation and denote ${\mathbb{\Omega}}\text {-q.i.}$  as ``$\simeq_1'$'', $\widehat{\mathbb{\Omega}} \text {-q.i.}$. as ``$\simeq_2'$'', and q.i. as ``$\simeq_3'$.''
\subsection{Quillen Adjunctions}
In this section, we recall the model structures on $\sS_0$ and $\sCC_{\F}^0$ used in Proposition 7.3.3 in \cite{RR22}. 
We promote these to the equivariant setting in Section 4. 

First, we recall the functors involved in the adjunctions. 
\begin{definition}\label{d2.3.1}
The \textit{simplicial coalgebra of chains} functor
$$\F[-]: \sS_0 \rightarrow \sCC_\F^0,$$
is defined on any $X \in \sS_0$ as the $\F$ vector space $\F[X]_n = \F[X_n]$ equipped with the coproduct 
$$\Delta: \F[X] \rightarrow \F[X] \otimes \F[X]$$
determined by
$$x \mapsto x \otimes x,$$
for any $x \in X_n$. 
\end{definition}
Note that this definition makes sense for any ring and if $X \in G\text{-}\sS_0$ we obtain an induced $G$-action on $\F[X]$.
The right adjoint of this functor is defined as follows. 
\begin{definition}\label{d2.3.2}
The \emph{functor of $\F$-points} is the functor
$$\mathcal{P}: \sCC_\F^0 \rightarrow \sS_0,$$
given for any $C \in \sCC_\F^0$ by the reduced simplicial set $$\mathcal{P}(C)_n = \operatorname{Hom}(\F, C_n).$$
\end{definition}
Note that
$$\operatorname{Hom}(\F, C_n) \cong \{ x \in C_n | \Delta_n(x) = x \otimes x, \epsilon_n(x) = 1 \},$$
with $\Delta_n$ and $\epsilon_n$ denoting the coproduct and counit of $C_n$, respectively. 
Again, if $C \in G\text{-}\sCC_\F^0$ we obtain an induced $G$-action on $\mathcal{P}(C)$.  
\begin{theorem}\label{t2.3.3}(Proposition 7.3.3 and Theorem 8.2.1 in \cite{RR22}) Let $\mathbb{F}$ be a field.
The adjunctions:
$$
\begin{gathered}
\mathbb{F}[-]:\left(\mathsf{sSet}_0, \simeq_1 \right) \rightleftarrows\left(\mathsf{sCoCoalg}_{\mathbb{F}}^0, \simeq_1' \right): \mathcal{P}, \\
\mathbb{F}[-]:\left(\mathsf{sSet}_0, \simeq_2 \right) \rightleftarrows\left(\mathsf{sCoCoalg}_{\mathbb{F}}^0, \simeq_2 '\right): \mathcal{P}, \\
\mathbb{F}[-]:\left(\mathsf{sSet}_0, \simeq_3 \right) \rightleftarrows\left(\mathsf{sCoCoalg}_{\mathbb{F}}^0, \simeq_3' \right): \mathcal{P},
\end{gathered}
$$
are Quillen adjunctions. 
Moreover, if $\F$ is algebraically closed, the functor $\F[-]$ is homotopically full and faithful in all three cases.
\end{theorem}
There are several ways to deem a Quillen adjunction as \emph{homotopically full and faithful}.
We deem a Quillen functor homotopically full and faithful if the derived unit map of the Quillen adjunction is is a weak equivalence for every cofibrant object.
So in this context the derived unit of $(\F[-], \mathcal{P})$ is in the class $\simeq_i$ for every cofibrant object. 

\section{Equivariant model category structures}
\begin{definition}
   Let $G$ be a group considered as a category with one object. The \textit{category of $G$-objects in $\sfC$}, denoted by $G$-$\sfC$, is defined as the category of functors from $G$ to $\sfC$ with natural transformations as morphisms.
\end{definition}
We review a criterion that guarantees the existence of an appropriate model structure on $G$-$\sfC$ given a sufficiently nice model structure on $\sfC$. 
We assume the group $G$ is a discrete simplicial group and $\sfC$ a cofibrantly generated model category, which is the case for $\sS_0$ or $\sCC_\F^0$ with any of the three model structures discussed. 
The desired model structure on $G$-$\sfC$ will be given in terms of fixed point objects.

\begin{definition}
    Given any subgroup $H$ of $G$, define the $H$-fixed point functor, denoted by \((-)^H\), as the composite \(G\text{-}\sfC \xrightarrow[]{\operatorname{restrict}} H\text{-}\sfC  \xrightarrow[]{\operatorname{lim}}  \sfC .\)
\end{definition}
We define the \emph{fixed point model structure} on $G$-$\sfC$, if it exists, as the model category structure determined by having weak equivalences (fibrations) be the maps $X \rightarrow Y$ such that, for every subgroup $H$ of $G$, the map $X^H \rightarrow Y^H$ is a weak equivalence (fibration) in $\sfC$.

Since model categories are by convention complete and cocomplete, any model category $\sfC$ may be thought as tensored and cotensored over the category of sets in the following way.
Let $X$ and $Y$ be objects of $\sfC$ and $S$ a set. 
Then there is a tensor and cotensor
\[S \otimes X = \coprod_S X \quad \text{and} \quad [S,Y] = \prod_S Y,\]
such that \(\operatorname{Hom}_\sfC( S\otimes X, Y ) \cong \operatorname{Hom}_{\s}(S, \operatorname{Hom}_\sfC(X,Y)) \cong \operatorname{Hom}_\sfC (X, [S,Y]).\)

Stephan's Proposition 2.6 in \cite{Ste16} provides a criteria for checking when the fixed point model structure exists based on the following so-called \emph{cellularity conditions}.

\begin{definition}\label{cellular}
    The $H$-fixed point functors $(-)^H: G\text{-}\sfC \rightarrow \sfC$ are said to be \textit{cellular}, or satisfy the \textit{cellularity conditions}, if for any subgroup $H$ of $G$ the following conditions hold: 
\begin{enumerate}
    \item the functor $(-)^H$ preserves pushouts of diagrams where one arrow is of the form 
    $$G/ K \otimes f : G/K \otimes X \rightarrow G/K \otimes Y,$$
    for some subgroup $K$ of $G$ and $f$ a generating cofibration of $\sfC$,
    \item the functor $(-)^H$ preserves filtered colimits of diagrams in $G$-$\sfC$,
    \item for any subgroup $K$ of $G$ and object $X$ in $\sfC$, the induced map 
    $$(G/H)^K \otimes X \rightarrow (G/H \otimes X)^K,$$
    is an isomorphism in $\sfC$.
\end{enumerate}
\end{definition}
We will make regular use of following theorems so we recall them here.
\begin{theorem}\label{t3.0.4}(Thereom 3.17 in \cite{Ste16}, Theorem 1.8, Chapter 3 in \cite{mandell2002})
    If $\sfC$ is a cofibrantly generated model category, with a generating set of cofibrations denoted $I$ and a generating set of acyclic cofibrations dentored $J$, such that $G$-$\sfC$ has cellular fixed point functors, then $G$-$\sfC$ is a cofibrantly generated model category with generating cofibrations
    $$I_G = \{G/H \otimes i \}_{i \in I, H \leq G}, $$
    generating acyclic cofibrations
    $$J_G = \{G/H \otimes j\}_{j \in J, H \leq G},$$ 
    and weak equivalences are maps of $G$-objects which are weak equivalences in $\sfC$ after passage to $H$ fixed points for all $H \leq G$.
\end{theorem}
We are also interested in the model category structure of the category of presheaves $\Og^{op}$-$\sfC$, where $\Og$ denotes the orbit category of the group $G.$
The category $\Og$ is the full subcategory of the category of $G$-sets with objects being orbits $G/H$ for $H$ a subgroup of $G$. Hence, $\Og^{op}$-$\sfC$ is the category of contravariant functors from $\Og$ to $\sfC$ with morphisms being natural transformations.

Recall Theorem 11.6.1 in \cite{hird}, namely that any category of diagrams in the cofibrantly generated model category $\mathsf{C}$ admits the projective model structure and is again cofibrantly generated.
Suppose that $\mathsf{C}$ has generating cofibration set $I$ and generating acyclic cofibration set $J$.
Then the category of orbit diagrams
$\Og^{op}$-$\mathsf{C}$ has generating cofibrations and generating cofibrations
$$I_{\Og^{op}}:=\{ \Og^{op}(G / H, -) \otimes i\}_{i \in I, H \leq G},$$
and generating acyclic cofibrations
$$J_{\Og^{op}}:=\{\Og^{op}(G / H, -) \otimes j\}_{j \in J, H \leq G} .$$
It is worth noting that, for a generating cofibration \(i: X \rightarrow Y\) in \(\mathsf{C}\), one can think of \(\Og^{op}(G / H, -) \otimes i\) as a natural transformation, from a functor defined by taking \(G/K\) to \((G/K)^H \otimes X\)  to another functor from to \(G/K\) to \((G/K)^H \otimes Y\). 
This natural transformation on components is a map from \((G/K)^H \otimes X\) to \((G/K)^H \otimes Y\). 

In this model structure, the weak equivalences and fibrations are the object-wise weak equivalences and fibrations.

\subsection{Equivariant Reduced Simplicial Sets}
We now deduce that the desired model structures exist on $G$-$\sS_0$ and $G$-$\sCC_\F^0$.
\begin{proposition}\label{p3.1.1}
For any of the three model structures in Theorem \ref{t2.1.3} on $\sS_0$, the corresponding fixed point model structure on $G$-$\sS_0$ exists. 
\end{proposition}
\begin{proof}
The cellularity conditions follow by Example 2.14 in \cite{Ste16}, so Theorem \ref{t3.0.4} applies.
\end{proof}
This proposition is  item (1) in Theorem \ref{one} and is an example (Example 1.1 and Example 2.14) provided in \cite{Ste16}.
These examples explain that the cellularity conditions hold for any presheaf category with a cofibrantly generated model structure having generating cofibrations being monomorphisms.
We denote $G$-$\sS_0$ with fixed point model structure as $(G$-$\sS_0, \text{\emph{fixpt}}(\simeq_i))$, where $\simeq_i$ is the notion of weak equivalence $\sS_0$ is equipped with. 

\begin{proposition}\label{p3.1.2}For any of the three model structures in Theorem \ref{t2.1.3} on $\sS_0$, 
the corresponding projective model structure on $\Og^{op}$-$\sS_0$ exists.
\end{proposition}
\begin{proof}
Note that this proposition follows from Proposition A.2.8.2 in \cite{Lur09}, since the model category structures $(\sS_0, \simeq_i)$ constructed in \cite{RR22} are left proper and combinatorial. 
\end{proof}
Again, we denote $\Og^{op}$-$\sS_0$ with projective model structure as 
$(\Og^{op}$-$\sS_0$,\text{\emph{proj}}$(\simeq_i))$, where $\simeq_i$ is the notion of weak equivalence $\sS_0$ is equipped with, and will refer to this as ``the presheaf category.''  
We conclude this subsection by proving a version Elmendorf's theorem.

\begin{theorem}\label{t3.1.3}
    There is a Quillen equivalence of model categories 
    $$ \Theta: (\Og^{op} \text{-} \sS_0, \text{proj}(\simeq_i))\rightleftarrows  (G\text{-} \sS_0, \text{fixpt}(\simeq_i)) : \Phi,$$
    where $\Theta(X) = X(G/e)$ and $\Phi(Y) = \underline{Y}$ with $\underline{Y}$ being the functor that takes $G/H$ to the reduced simplicial set $Y^H$.
\end{theorem}
\begin{proof}
We know that the fixed point functors $(-)^H$ preserve weak equivalences and fibrations by the definition of the model structure on $G$-$\sS_0$, so correspondingly $\Phi$ will also preserve weak equivalences and fibrations. 

To show the desired Quillen equivalence, we want show that for every cofibrant object $X$ in $\Og^{op}$-$\sS_0$ and fibrant object $Y$ in $G$-$\sS_0$ that 
\(\alpha: \Theta(X) \rightarrow Y \) is a weak equivalence in \( G\text{-}\sS_0\),
if and only if \( \beta: X \rightarrow \Phi(Y) \)  is a weak equivalence in \(\Og^{op}\text{-}\sS_0\). 
Using the proof of Theorem 2.10 in \cite{Ste16}, we know that for any cofibrant $X$ in $\Og^{op}$-$\sS_0$ the unit of the adjunction \((\Theta, \Phi)\), \(\eta: X \rightarrow \Phi(\Theta(X))\),
is an isomorphism. 

Given $\alpha:\Theta(X) \rightarrow Y $, we know that $\beta: X \rightarrow \Phi(Y)$ gives rise to the map $\beta^H= \beta(G/H),$
$$\beta(G/H): X(G/H) \rightarrow \Phi((Y)(G/H)) = Y^H,$$
and this map can be factored as follows 
$$X(G/H) \xrightarrow[]{\eta_{G/H}} \Phi(\Theta (X(G/H))) = X(G/e)^H \xrightarrow[]{\alpha^H} Y^H.$$
By the 2-out-of-3 axiom of being a model category, this shows that $\alpha$ is a weak equivalence if and only $\beta$ is a weak equivalence. 
\end{proof}

\subsection{Equivariant Connected Simplicial Coalgebras}  We have analogous statements to those of the previous section for equivariant simplicial cocommutative coalgebras. 

\begin{proposition}\label{p3.2.1}
For any of the three model structures on $\sCC_\F^0$, $G$-$\sCC_\F^0$ admits the $G$-model structure 
admits the fixed point model structure.

\end{proposition}
This proposition is item (2) in Theorem \ref{one} and holds for the same reasoning given for $G$-$\sS_0$ in Proposition \ref{p3.1.1}.
As before, we are also interested in the model structure on the presheaf category. 
\begin{proposition}\label{p3.2.2}
For any of the three model structures on $\sCC_\F^0$,
$\Og^{op}$-$\sCC_\F^0$ admits the projective model structure.
\end{proposition}
It was shown in Theorem  7.3.1 in \cite{RR22} that the model structures $(\sCC_\F^0, \simeq_i')$ are left proper and combinatorial, so the above statement follows from Proposition A.2.8.2 in \cite{Lur09}. Lastly, we have the following result comparing the two model structures,  whose proof can be adapted from the proof of Theorem \ref{t3.1.3}.

\begin{theorem}\label{t3.2.3}
    There is a Quillen equivalence of model categories 
    $$ \Theta: (\Og^{op} \text{-} \sCC_\F^0,  \text{proj}(\simeq_i') )  \rightleftarrows  (G\text{-} \sCC_\F^0,  \text{fixpt}(\simeq_i') )  : \Phi,$$
    where $\Phi(A) = A(G/e)$ and $\Theta(B) = \underline{B}$ with $\underline{B}$ being the functor that takes $G/H$ to the connected simplicial cocommutative coalgebra $B^H$.
\end{theorem}
\section{Equivariant Chains Functors}

If $X \in G\text{-}\sS_0$ then we have an action $G \times X \rightarrow X$ and for $G$ a discrete simplicial group we have that $G_n \times X_n \rightarrow X_n$. 
Therefore, we get an induced action on $\F[X]$, which commutes with the coproduct.
Also, by the Propositions \ref{p3.1.1} and \ref{p3.1.2}, we know we have fixed point model structure on $G$-$\sS_0$ and $G$-$\sCC_\F^0$, so we define an equivariant version of the simplicial coalgebra of chains functor and its adjoint between these categories
$$\F_G[-]: G\text{-}\sS_0 \rightleftarrows G\text{-}\sCC_\F^0: \mathcal{P}_G,$$
to be the same as $\F[-]$ and $\mathcal{P}$, respectively, but now on $G$-objects, meaning that for $\F_G[X] :=\F[X]$ and $\mathcal{P}_G[X] :=\mathcal{P}[X].$
As to be expected the these functors preserve the $G$-action given $X \in G$-$\sS_0$ and  $C \in G$-$\sCC_\F^0$
$$\F_G[G \times X] = \F[G \times X] = G\times \F[X] = G \times \F_G[X],$$
$$\mathcal{P}_G(G \times C) = \mathcal{P}(G \times C) = G\times \mathcal{P}(C) = G\times \mathcal{P}_G(C).$$

As a consequence, $(\F_G[-], \mathcal{P}_G)$ will also form an adjunction.

If $\underline{X} \in \Og^{op}\text{-}\sS_0$, we define the functor 
$$\F_{\Og^{op}}[-]: \Og^{op}\text{-}\sS_0 \rightarrow \Og \text{-}\sCC_\F^0,$$
$$\underline{X} \mapsto \F_{\Og^{op}}[X],$$
where the new object in $\Og^{op}$-$\sCC_\F^0$ is the composite functor determined by $\underline{X}$ and $\F[-]$.
Correspondingly, given $\underline{A}\in \Og^{op}$-$\sCC_\F^0$ we can define the functor
$$\mathcal{P}_{\Og^{op}}: \Og^{op} \text{-}\sCC_\F^0 \rightarrow \Og^{op}\text{-}\sS_0,$$
$$\underline{A} \mapsto \mathcal{P}_{\Og^{op}}(A),$$
where again the new object in $\Og^{op}$-$\sS_0$ is the composite functor determined by $\underline{A}$ and $\mathcal{P}$.
As way to organize the functors and categories in question we present the diagram below: 
\begin{eqnarray}\label{commsquare}
    \begin{tikzcd} 
{(G\text{-}\sS_0, \text{\emph{fixpt}}(\simeq_i))}  \arrow[rrr, "{\F_G[-]}", shift left= 2] \arrow[dd, "\Phi",shift right=-2 ] &  &  &
{(G\text{-}\sCC_\F^0, \text{\emph{fixpt}}(\simeq_i))} \arrow[lll, "\mathcal{P}_{G}"]   \arrow[dd, "\Phi",shift right=-2] \\
& & &   \\
{(\Og^{op}\text{-}\sS_0, \text{\emph{proj}}(\simeq_i) )}\arrow[uu, "\Theta", shift left] \arrow[rrr, "{\F_{\Og^{op}}[-]}", shift left=2]                  &  &  & 
{(\Og^{op}\text{-}\sCC_\F^0, \text{\emph{proj}}(\simeq_i)).} \arrow[lll, "\mathcal{P}_{\Og^{op}}"] \arrow[uu, "\Theta", shift left]    
\end{tikzcd}
\end{eqnarray}
for the $\simeq_i$ in the previous sections and the Quillen equivalences in Theorems \ref{t3.1.3} and \ref{t3.2.3}.
Notice that the square that commutes on the nose is the following
$$\F_G[-] =\Theta \circ \F_{\Og^{op}}[-] \circ \Phi.$$

With this having been established, we can state the main result, namely item (3) in Theorem \ref{one}.
\begin{theorem}\label{t4.0.1}
For any field, $\F$, $\F_G[-]$ and $\F_{\Og^{op}}[-]$ are left Quillen functors. 
\end{theorem}

\begin{proof}
We show that $\F_G[-]$ and $\F_{\Og^{op}}[-]$ preserve cofibrations and acyclic cofibration. 
We know that $\sS_0$ is cofibrantly generated by \cite{RR22}, so we consider the following generating set of cofibrations and acyclic cofibration 
$$I = \{ i: X \hookrightarrow Y \mid  X, Y \text{ finite}\} ,$$
$$J= \{ j: X \xhookrightarrow{\sim} Y \mid  X, Y \text{ finite}\} .$$
By Theorem \ref{t3.0.4}, we know that
$G$-$\sS_0$ is cofibrantly generated with generating sets: 
$$I_G = \{ G/H \times i \}_{i \in I, H \leq G} = \left\{ \coprod_{G/H} i \right\}_{i \in I, H \leq G},$$
$$J_G = \{ G/H \times  j\}_{j \in J, H \leq G} = \left\{ \coprod_{G/H} j\right\}_{j \in J, H \leq G}.$$
On cofibrations and acyclic cofibrations the action of $G$ is given by permuting copies of the same map as the indexing is given by the orbits. 

Correspondingly, we know, by \cite{RR22}, that is $\sCC_\F^0$ cofibrantly generated. 
Consider generating set of cofibrations and acyclic cofibrations 
$$I' = \{ i': A \hookrightarrow B \mid  A, B \text{ finite  as $\F$-vector spaces} \}, $$
$$J' = \{ j': A \xhookrightarrow{\sim} B \mid  A, B \text{ finite  as $\F$-vector spaces}\}. $$
Again, by Theorem \ref{t3.0.4}, we know that 
$G$-$\sCC_\F^0$ is cofibrantly generated with generating sets: 
$$I'_G = \{ \F[G/H] \otimes i' \}_{i' \in I', H \leq G}, $$
$$J'_G = \{ \F[G/H] \otimes  j'\}_{j' \in J', H \leq G}. $$
Observe that $\F[G/H] \otimes \F[i]$ is a map in $I'_G$, because for $i: X \rightarrow Y \in I$ then $\F[i] \in I'$, since $\F[X]$ and $\F[Y]$ will be finite as $\F$-vector spaces. 
Moreover, $\F[-]$ preserves cofibrations and acyclic cofibrations. So $\F_G[G/H \times i]$ is a cofibration for any $i \in I$, since $\F[G/H] \otimes \F[i] = \F_G[G/H\times i].$
Since $\F_G[-]$ preserves generating cofibrations, we can conclude that $\F_G[-]$ preserves all cofibrations because they are directed colimits of generating cofibrations. 
An analogous argument holds for acyclic cofibrations. 

To conclude that $\F_{\Og^{op}}[-]$ preserve cofibrations and acyclic cofibrations, recall that the category of orbit diagrams
$\Og^{op}$-$\sS_0$ has generating cofibrations,
$$I_{\Og^{op}}:=\{ \Og^{op}(G / H, -) \otimes i\}_{i \in I, H \leq G},$$
and generating acyclic cofibrations
$$J_{\Og^{op}}:=\{\Og^{op}(G / H, -) \otimes j\}_{j \in J, H \leq G} .$$
Provided a subgroup $K \leq G$, these generating sets become: 
$$I_{\Og^{op}}(G/K):=\{ \Og^{op}(G / H, G/K) \otimes i\}_{i \in I, H,K \leq G} = \left\{ \bigoplus_{(G/K)^H} i \right\}_{i \in I, H, K \leq G},$$
and generating acyclic cofibrations
$$J_{\Og^{op}}(G/K):=\{\Og^{op}(G / H, G/K) \otimes j\}_{j \in J, H,K \leq G} = \left\{ \bigoplus_{(G/K)^H} j \right\}_{i \in I, H, K \leq G}.$$
Note, as before, the action of $G$ on acyclic cofibrations is given by permuting copies of the same map with indexing given by the $(G/K)^H$. 

Observe that for any $i \in I$, the map $\oplus_{(G/K)^H} \F[i]$ is a cofibration in $\Og^{op}$-$\sCC_\F^0$ with domain and codomain of $\F[i]$ being finite dimensional vector spaces. 
Since 
$$\oplus_{(G/K)^H} \F[i] = \F [\sqcup_{(G/K)^H}i] =
\F_{\Og^{op}} [\oplus_{(G/K)^H}i],$$
we have $\F_{\Og^{op}}[-]$ preserves generating cofibrations so this is enough to conclude that it preserves all cofibrations and again an analogous argument holds for acyclic cofibrations. 
\end{proof}

A natural question to ask is how far away is the adjunction $(\F_G[-], \mathcal{P}_G)$ from being a Quillen equivalence. 
In the non-equivariant case \cite{RR22}, we have that when the field $\F$ is algebraically closed the functor $\F[-]$ is homotopically full and faithful. 
Now, using the functors in the equivariant adjunction and prove a similar statement. 
\begin{proposition}\label{p4.0.2}
If $\F$ is an algebraically closed field, then $\F_{\Og^{op}}[-]$ is homotopically full and faithful given any of the three model structures. 
\end{proposition}

\begin{proof} 
We show that the derived unit of the adjunction $(\F_{\Og^{op}}[-], \mathcal{P}_{\Og})$, denoted $\eta_{\Og^{op}}$, is a weak equivalence. 
Consider a cofibrant object $\underline{X}$ in $\Og^{op}$-$\sS_0$ and now the goal is to show
$$\eta_{\Og^{op}}: \underline{X} \rightarrow \mathcal{P}_{\Og^{op}}(\F_{\Og^{op}} [\underline{X}]), $$
this is a weak equivalence in $\Og^{op}$-$\sS_0$, meaning its a weak equivalence object-wise for every cofibrant object. 
So given $G/H \in \Og$, we want to investigate whether the map below is a weak equivalence,
$$\eta_{\Og^{op}} \mid_{G/H} : \underline{X}(G/H) \rightarrow \mathcal{P}_{\Og^{op}}(\F_{\Og^{op}}[\underline{X}(G/H)]^f),$$
where $-^f$ denotes the cofibrant replacement.
Note that $\eta_{\Og^{op}} \mid_{G/H}$ coincidences with the derived unit of $(\F[-], \mathcal{P})$ which we know is a weak equivalence by Theorem \ref{t2.3.3}.
Hence, we conclude that $\F_{\Og^{op}}[-]$ is homotopically full and faithful, as desired. 
\end{proof}
Note that $\eta_{\Og^{op}}$ will be a weak equivalence for any $G/H \in \Og$ and for any $\underline{X}$, which is perhaps not cofibrant, since all objects in $\sS_0$ are fibrant/cofibrant. 
Given this we get the following corollary, which provides a relationship between $\Theta$ and $\mathcal{P}$ and between $\Phi$ and $\F[-]$.  

\begin{corollary}\label{c4.0.3}
At the level of the appropriate homotopy categories (i.e. after deriving functors with respect to notions \(\simeq_i\) and \(\simeq_i'\)) we have
\begin{enumerate}
    \item $\Phi \circ \mathcal{P}_G \cong \mathcal{P}_{\Og^{op}} \circ \Phi$ if and only if $\mathcal{P}_G \circ \Theta \cong \Theta \circ \mathcal{P}_{\Og^{op}}$.
    \item 
$\Theta \circ \F_{\Og^{op}} [-] \cong \F_G[-] \circ \Theta$ if and only if
$\F_{\Og^{op}} [-] \circ \Phi \cong \Phi \circ  \F_G[-] $.
\end{enumerate}
\end{corollary}
\begin{proof}
This follows from the general fact that given a diagram of functors
$$\begin{tikzcd}
\sfC \arrow[rrr, "F", shift right=2] \arrow[dd, "\Phi",shift right=-2 ] &  &  &
\mathsf{D}  \arrow[dd, "\Phi",shift right=-2 ] \\
& & &   \\
\sfC' \arrow[uu, "\Theta", shift left] \arrow[rrr, "F'", shift left=2]                  &  &  & 
\mathsf{D}'. \arrow[uu, "\Theta", shift left]    
\end{tikzcd}$$
where the two pairs of vertical arrows define an equivalence of categories, we have $F' \circ \Phi \cong \Phi \circ F$ if and only if $F \circ \Theta \cong \Theta  \circ F'$. The corollary follows by applying this fact to Theorem \ref{t4.0.1}, which tells us that the horizontal pairs of functors in the square \ref{commsquare} are Quillen adjunctions. 
    \end{proof}
 
\begin{theorem}\label{t4.0.4}
    If $\F$ is an algebraically closed field, then the derived functor
    $$\F_{G}[-]: \operatorname{Ho}(G\text{-}\sS_0) \rightarrow \operatorname{Ho}(G\text{-}\sCC^0_\F),$$
    is full and faithful given any of the three model structures. 
\end{theorem}
\begin{proof}
    The task at hand, is to show that derived unit of the adjunction $(\F_G[-], \mathcal{P}_G)$, denoted $\eta_G$, is a weak equivalence. 
    Consider $\Phi (S) $ for some $S\in G$-$\sS_0$ then we know by \ref{p4.0.2}
    \begin{align*} 
    \Phi (S) \cong & \mathcal{P}_{\Og^{op}} \circ \F_{\Og^{op}}[-] \circ \Phi (S) \\
    \Phi (S) \cong & \mathcal{P}_{\Og^{op}} \circ \Phi \circ \F_{G}[-] (S) \\
    \Phi (S) \cong & \Phi \circ \mathcal{P}_{G} \circ \F_{G}[-] (S) \\
    \Theta \circ \Phi (S) \cong & \Theta \circ \Phi \circ \mathcal{P}_{G} \circ \F_{G}[-] (S) \\
    S \cong & \mathcal{P}_{G} \circ \F_{G}[-] (S) = \mathcal{P}_{G} \circ \F_{G}[S] \\
    \end{align*}
    as desired. 
    Note the isomorphisms hold due to Corollary \ref{c4.0.3}.
\end{proof}

This theorem is item (4) in Theorem \ref{one} and illustrates the usefulness of Elmendorf's theorem and the perspective of modeling equivariant homotopy theory in terms of diagrams. 

The remaining portion of item (4) in Theorem \ref{one} is contained in the next section.

\section{Perfect Fields}
In the previous section, we assumed that $\F$ was algebraically closed field, to apply certain structural results about coalgebras over algebraically closed fields. We now work in the more general setting of perfect fields. 
We will set up the appropriate framework and cite the necessary results from Section 8 in \cite{RR22}.

Let $\F$ be a perfect field denote by $\overline{\F}$ the algebraic closure of $F$, and by $\G$ the absolute Galois group of $\F.$ We denote the category of $\G$-sets by $\s(\G)$ and the category of $\G$-simplicial sets by $\sS(\G)$.

By Theorem 8.3.2. in \cite{RR22}, there are three left proper combinatorial model category structures on  $\sS(\G)$ where the cofibrations are the monomorphisms and the weak equivalences are maps in $\sS(\G)$ such that the underlying map in $\sS$ is a categorical $\F$-equivalence (respectively, $\pi_1$-$\F$-equivalence, and $\F$-equivalence). 
In line with the previous notation of this text, we denote the categories with the respective model structure as $(\sS_0(\G),\simeq_i\!(\G))$.
This model category structure is sometimes referred to as the ``naive model category structure'' (note we are not taking fixed points in this context). 

Since, for each $i$, $(\sS(\G),\simeq_i\!(\G)) $ is a cofibrantly generated model category with generating cofibrations being monomorphism we obtain the following proposition whose proof uses the same reasoning as that of Proposition \ref{p3.1.1} and \ref{p3.2.1}.
\begin{proposition}\label{p5.0.1}
For discrete group $G$ and any of the three model structures on $\sS(\G)_0$, the fixed point model structure on $G$-$\sS(\G)_0$ exists. 
\end{proposition}
In an effort to streamline notation, lets denote this new model structure as 
$$(G\text{-}\sS(\G)_0, \text{\emph{fixpt}}(\simeq_i\!(\G))).$$
With this proposition, we get a version of Elmendorf's theorem for this setting:
\begin{theorem}\label{t5.0.2}
    There is a Quillen equivalence of model categories 
    $$ \Theta: ( \Og^{op} \text{-} \sS(\G)_0, \text{proj}(\simeq_i\!(\G)))  \rightleftarrows  (G\text{-} \sS(\G)_0,  \text{fixpt}(\simeq_i\!(\G))): \Phi,$$
    where $\Phi(S) = S(G/e)$ and $\Theta(T) = \underline{T}$ with $\underline{T}$ being the functor that takes $G/H$ to the reduced simplicial set $T^H$, which is also equipped with an action of $\G$.
\end{theorem}

\begin{proof}
  As with Theorem \ref{t3.2.3}, the proof of this result can be adapted from the proof of Theorem \ref{t3.1.3}.
\end{proof}
For any $\G$-set $S,$ the $\G$-fixed points of the $\overline{\F}$-coalgebra $\overline{\F}[S]$ form an $\F$-coalgebra. 
So, we can define the functor
$$\overline{\F}[-]^{\G}: \s(\G) \rightarrow \CoCo_{\F},$$
which admits a right adjoint induced by $\mathcal{P}$, defined as follows: 
$$\mathcal{P}_{\G}: \CoCo_{\F} \rightarrow \s(\G),$$
$$A \mapsto \CoCo_{\F}(\overline{\F}, A \otimes_{\F} \overline{\F}).$$
Additionally, for any $\G$-set $S$, there is a canonical isomorphism of $\F$-coalgebras 
$$\overline{\F}[S]^{\G} \otimes \overline{\F} \cong \overline{\F}[S].$$
Hence, we get that the unit of the adjunction ($\overline{\F}[-]^{\G}, \mathcal{P}_{\G})$,
$$S \mapsto \mathcal{P}_{\G}( \overline{\F}[S]^{\G}),$$
is a natural isomorphism. 
We now show an analogous result to Theorem \ref{t4.0.4}, which will make use of the adjunction 
$$\overline{\F}_G[-]^{\G}: G\text{-}\sS(\G)_0 \leftrightarrows G\text{-}\sCC^0_{\F} : \mathcal{P}_{\G, G}, $$
where $\overline{\F}_G[-]^{\G}$ is factored as follows:
$$\begin{tikzcd}
G\text{-}\sS(\G)_0 \arrow[rd, "{\overline{\F}_G[-]}"'] \arrow[rr, "{\overline{\F}_G[-]^\G}"] &   & G\text{-}\sCC_\F^0 \\   
& G\text{-}\sCC_\F^0\arrow[ru, "(-)^{\G}"'], &    
\end{tikzcd}$$
and $\mathcal{P}_{\G, G}$ is the induced functor on $G$-objects of $\sCC_{\F}^0$ and $\sS(\G)_0$. We have an induced functor 
$$\overline{\F}_{\Og^{op}}[-]^\G:  \Og^{op}\text{-}\sS(\G)_0 \rightarrow \Og^{op}\text{-}\sCC^0_{\F}, $$
at the level of presheaves over $\Og^{op}$ and and induced right adjoint $\mathcal{P}_{\G, \Og^{op}}$.

\begin{proposition}\label{p5.0.3}
    Given any of the three model structures and a perfect field, $\F$, then $\overline{\F}_{\Og^{op}}[-]^{\G}$ is homotopically full and faithful.
\end{proposition}

\begin{proof} 
We show the derived unit of the adjunction $(\overline{\F}_{\Og^{op}}[-]^{\G}, \mathcal{P}_{\G, \Og^{op}})$, denoted $\eta_{\G, \Og^{op}}$, is a weak equivalence. 
To see this we show that for any cofibrant object $\underline{X}$ in $\Og^{op}$-$\sS_0(\G)$, 
      $$\eta_{\G, \Og^{op}}: \underline{X} \rightarrow \mathcal{P}_{\G, \Og}(\overline{\F}_{\Og} [\underline{X}]^{\G}),$$
 is a weak equivalence in $\Og$-$\sS_0(\G)$, meaning it is a weak equivalence object-wise.
 Given $G/H \in \Og$, consider
   $$\eta_{\G, \Og^{op}} \mid_{G/H} : \underline{X}(G/H) \rightarrow \mathcal{P}_{\G, \Og^{op}}((\F_{\G, \Og^{op}}[\underline{X}(G/H)]^{\G})^f),$$
where $(-)^f$ again denotes the cofibrant replacement functor.
Note that $\eta_{\Og^{op}} \mid_{G/H}$ coincidences with the derived unit of $(\overline{\F}[-]^{\G}, \mathcal{P}_{\G})$, which is a weak equivalence. 
This is because in the proof of Theorem 8.3.4 in \cite{RR22} it was shown that $(\overline{\F}[-]^{\G}, \mathcal{P}_{\G})$ satisfies Proposition 6.3.3 in \cite{RR22}. \end{proof}
We now conclude the following result.
\begin{theorem}\label{6.2}
    Let $\F$ be a perfect field, $\overline{\F}$ its algebraic closure, and $\G$ the absolute Galois group.
    Then given any of the three model structures, the derived functor
    $$\overline{\F}_G[-]^{\G}: \operatorname{Ho}(G\text{-}\sS(\G)_0) \rightarrow \operatorname{Ho}(G\text{-}\sCC_{\F}^0),$$
    is full and faithful.
\end{theorem}
\begin{proof}
Consider the diagram below where the vertical functor are given by those in Theorems \ref{t5.0.2} and \ref{t3.2.3}:
$$\begin{tikzcd}
G\text{-}\sS(\G)_0 \arrow[rrr, "{\overline{\F}_G[-]^{\G}}", shift right=2] \arrow[dd, "\Phi"', shift right] &  &  &
G\text{-}\sCC_{\F}^0  \arrow[dd, "\Phi"', shift right] \\
& & &   \\
\Og^{op}\text{-}\sS(\G)_0 \arrow[uu, "\Theta"', shift right] \arrow[rrr, "{\overline{\F}_{\Og}[-]^{\G}}", shift left=2]                  &  &  & 
\Og^{op}\text{-}\sCC_{\F}^0.  \arrow[uu, "\Theta"', shift right]    
\end{tikzcd}$$

Using Theorem \ref{t4.0.4} and the analogous corollary to Corollary \ref{c4.0.3} in this setting, we show that unit of the adjunction $(\overline{\F}_G[-]^\G, \mathcal{P}_{\G,G})$, denoted $\eta_{\G, G}$, is a weak equivalence. 
Consider $\Phi (S) $, for some $S\in G$-$\sS_0(\G)$, then by \ref{p4.0.2}, at the level of homotopy categories we have
    \begin{align*} 
    \Phi (S) \cong & \mathcal{P}_{\G, \Og^{op}} \circ \overline{\F}_{\Og}[-]^{\G} \circ \Phi (S) \\
     \Phi (S) \cong & \mathcal{P}_{\G, \Og^{op}} \circ \Phi \circ \overline{\F}_{G}[-]^{\G} (S) \\
      \Phi (S) \cong & \Phi \circ \mathcal{P}_{\G, G} \circ \overline{\F}_{G}[-]^{\G} (S) \\
      \Theta \circ \Phi (S) \cong & \Theta \circ \Phi \circ \mathcal{P}_{\G, G} \circ \overline{\F}_{G}[-]^{\G} (S) \\
       S \cong & \mathcal{P}_{\G, G} \circ \overline{\F}_{G}[-]^{\G} (S) = \mathcal{P}_{\G, G} \circ \overline{\F}_{G}[S]^{\G}, \\
 \end{align*} as desired. 
 The isomorphisms hold by the proof of Corollary \ref{c4.0.3} when one takes $F$ to be $\overline{\F}_{G}[-]^{\G}$ or $\mathcal{P}_{\G, G}$ and $F'$ to be 
 $\overline{\F}_{\Og}[-]^{\G}$ or $\mathcal{P}_{\G, \Og}$, respectively. 
\end{proof}
The following corollary measures the failure of $\F_G[-]$ to be homotopically full and faithful. 
This forms part of item (4) in Theorem \ref{one}.
\begin{corollary} \label{homotopyfixedpoints}
Let $\F$ be a perfect field with algebraic closure $\F \subset \overline{\F}$ and  absolute Galois group $\G$.
The derived unit transformation of the Quillen adjunction 
$$\F_{G}[-]: G\text{-}\sS_0 \rightleftarrows G\text{-}\sCC_{\F}^0: P_G, $$
for any of the three model structures prescribed in this text, is canonically identified with the derived unit transformation of the Quillen adjunction
$$\delta_G: G\text{-}\sS_0 \rightleftarrows G\text{-}\sS(\G)_0: (-)^{\G},$$
where $\delta_G$ is the functor that induces the trivial action of the Galois group.   
\end{corollary}
This above statement says that the derived unit transformation is identified, for each $\simeq_i$ and corresponding $\simeq_i'$, with the canonical map into the homotopy $\G$-fixed points $X \rightarrow (\delta_G (X))^{h\G}$, where $(-)^{h\G}$ is interpreted in the appropriate way in each model category. 

\begin{remark} Finally, we describe the connection to the rational homotopy theory of simply connected spaces. We first analyze the derived unit transformation in Corollary \ref{homotopyfixedpoints} in the non-equivariant case. 
Let $\G$ be the absolute Galois group of $\mathbb{Q}$ and suppose $X$ is a simply connected reduced simplicial set. Denote by $\delta X$ the simplicial set $X$ considered as a $\G$-simplicial set with the trivial action. 
Then the homotopy fixed points functor \[X \mapsto (\delta X)^{h\G},\] interpreted in the model category structure on $\mathsf{sSet}(\G)_0$ that has $\mathbb{Q}$-equivalences as weak equivalences, is precisely the rational localization of $X$. 
This follows since $\mathbb{G}$ is a profinite group, the rational homology of the classifying space $B\G$ satisfies $H_i(B\G;\mathbb{Q})=0$ for $i>0$. In fact, if $X$ is a rational space, then the map
\[ X=\text{Maps}(*,X) \to \text{Maps}(B\G,X)= (\delta X)^{h\G},\]
is a $\mathbb{Q}$-equivalence since $B\G \to *$ is. 
This means that the derived unit of the non-equivariant version of Corollary \ref{homotopyfixedpoints} (in the context of $\mathbb{Q}$-equivalences) is the rational localization when applied to a simply connected simplicial set.

In the $G$-equivariant setting, using the interpretation of $G$-$\mathsf{sSet}(\G)_0$ in terms of the category of diagrams $\mathcal{O}^{op}_G$-$\mathsf{sSet}(\mathbb{G})_0$, we obtain that the $G$-equivariant rationalization of a nilpotent $G$-simplicial set $X$ may be functorially recovered from its simplicial coalgebra of chains $\mathbb{Q}_G[X]$. 

\end{remark}

\printbibliography

\end{document}